\documentclass{amsart}
\usepackage{amssymb}
\usepackage{amsmath}
\usepackage{tikz-cd}
\usepackage{amsthm}
\theoremstyle{plain}

\theoremstyle{plain}
\newtheorem{thm}{Theorem}
\newtheorem{prop}{Proposition}
\newtheorem{lem}{Lemma}

\usepackage{geometry}
\theoremstyle{definition}

\newtheorem{rem}{Remark}

\newtheorem{definition}{Definition}

\begin{document}
\setcounter{page}{1}

\title{ Lattice structure on bounded homomorphisms between topological lattice rings}
\author[Omid Zabeti ]{Omid Zabeti}

\address{ Department of Mathematics, Faculty of Mathematics, University of Sistan and Baluchestan, P.O. Box: 98135-674, Zahedan, Iran.}
\email{{o.zabeti@gmail.com}}



\date{Received: xxxxxx; Revised: yyyyyy; Accepted: zzzzzz.}

\begin{abstract}
Suppose $X$ is a locally solid lattice ring. It is known that there are three classes of bounded group homomorphisms on $X$ whose topological structures make them again topological rings. In this note, we consider lattice structure on them; more precisely, we show that, under some mild assumptions, they are locally solid lattice rings.
\end{abstract}
\maketitle





\section{Introduction and Preliminaries}
Let us start with some motivation. Topological rings usually appear in many contexts in  functional analysis. The ring of all continuous functions on a topological space; where the topology is given by pointwise convergence, the integers with discrete topology, the ring of all of matrices with entries in a topological ring; where the topology is given by pointwise convergence are all examples of topological rings. Many of these examples have also lattice structures. So, topological lattice rings come readily to mind as an interesting subject to study in the category of all rings. Furthermore, when we have topological lattice algebraic structures, it is a natural and interesting direction to investigate functions ( homomorphisms) which respect topological, lattice and algebraic structures.

The concept of a lattice group ($\ell$-group, for short) was firstly investigated in \cite{B,C}. In addition, topological $\ell$-groups as an extension of topological Riesz spaces are appeared in \cite{S1,S2}, at first. Although, Riesz spaces are widely investigated in many directions for decades, lattice groups are rarely considered in the literatures; only recently, a comprehensive reference announced regarding basic properties of topological $\ell$-groups ( see \cite{H} for more details).

Nevertheless, the notion of a lattice ring ( $\ell$-ring) is even considered less than $\ell$-groups in the contexts. To our best knowledge, it is initially investigated in \cite{BP,J}. The situation got stricter while adding topological notion to them; the earliest special literature is \cite{W}.

Note that since topological $\ell$-groups are a generalization of topological Riesz spaces which contain many known and applicable objects such as Banach lattices and examples therein, they are investigated in more details at least in the contexts so that topological $\ell$-rings seem to be largely unexplored with respect to  topological $\ell$-groups. On the other hand, topological rings arise almost in many directions of topological fields; for example, the completion of a topological field is always a topological ring. Moreover, the set of all real continuous functions on a Hausdorff topological space, the set of all matrices defined on a field, are examples of rings which are widely useful in the literatures. So, it is of independent interest to discover different directions of rings such as topological and order notions; topological and order aspects are considered in several contexts, separately ( see \cite{AGM, J, S, W}, for example) but using both order and topological ones have been investigated not so much.

In \cite{Z1}, Mirzavaziri and the author considered three non-equivalent classes of bounded group homomorphisms on a topological ring and endowed them with appropriate topologies which make them again topological rings. Now, suppose $X$ is a locally solid $\ell$-ring. In this note, our attempt is to consider lattice structures on these classes of bounded homomorphisms. In fact, we show that under some mild hypotheses, they configure locally solid $\ell$-rings.

For  recent progress on topological $\ell$-groups as well as basic expositions on these notions, see \cite{H}. Finally, for undefined terminology, general theme about $\ell$-rings and the related subjects, we refer the reader to \cite{J}.

Let us first, recall some required notions and terminology. Suppose $X$ is a topological ring. A set $B\subseteq X$ is called bounded if for each zero neighborhood $W\subseteq X$, there is a zero neighborhood $V\subseteq X$ such that $VB\subseteq W$ and $BV\subseteq W$. Now, assume that $G$ is a topological group; a subset $B\subseteq G$ is said to be bounded if for each neighborhood $U$ at the identity, there is a positive integer $n$ with $B\subseteq nU$.

By a topological lattice group ( $\ell$-group), we mean an abelian topological group which is also a lattice at the same time such that the lattice operations are continuous with respect to the assumed topology. A topological lattice ring ( $\ell$-ring) is a topological ring which is simultaneously an $\ell$-group such that the multiplication and order structure are compatible via the inequality $|x\cdot y|\leq |x|\cdot|y|$; for more details, we refer the reader to \cite{H}.

A Birkhoff and Pierce ring ( $f$-ring) is a lattice ordered ring with this property: $a\wedge b=0$ and $c\geq 0$ imply that $ca\wedge b=ac\wedge b=0$. For ample facts regarding this subject, see \cite{J}. It was initially presented by Birkhoff and Pierce in \cite{BP} to illustrate some understandable examples in lattice ring theory and apparently, it turned out to have many interesting and fruitful tools among the category of lattice rings.

An $\ell$-group $G$ is called {\bf Dedekind complete} if every non-empty bounded above subset of $G$ has a supremum. $G$ is {\bf Archimedean} if $nx\leq y$ for each $n\in \Bbb N$ implies that $x\leq 0$. One may verify easily that every Dedekind complete $\ell$-group is Archimedean. In this note, all topological groups are considered to be abelian. A subset $S\subseteq G$ is called {\bf solid} if $x\in G$, $y\in S$, and $|x|\le |y|$ imply that $x\in S$. Topological $\ell$-group $(G,\tau)$ is said to be locally solid if $\tau$ contains a base of neighborhoods at identity consists of solid sets. $S$ is said to be {\bf order bounded} if it is contained in an order interval.

Suppose $G$ is a topological $\ell$-group. A net $(x_{\alpha})\subseteq G$ is said to be {\bf order} convergent to $x\in G$ if there exists a net $(z_{\beta})$ ( possibly over a different index set) such that $z_{\beta}\downarrow 0$ and for every $\beta$, there is an $\alpha_0$ with $|x_{\alpha}-x|\leq z_{\beta}$ for each $\alpha\ge \alpha_0$. A subset $A\subseteq G$ is called {\bf order closed} if it contains limits of all order convergent nets which are lying on $A$.

Keep in mind that topology $\tau$ on a topological $\ell$-group $(G,\tau)$ is referred to as {\bf Fatou} if it has a local basis at the identity consists of solid order closed neighborhoods.

Suppose $G$ and $H$ are $\ell$-groups. A homomorphism $T:G\to H$ is said to be {\bf positive} if it maps positive elements of $G$ into positive ones in $H$.

Now, we recall some definition we need in the sequel ( see \cite{Z1} for further notifications about these facts). It should be mentioned here that in \cite{Z1}, the authors used the notion $\sf B(X,Y)$ for rings of all bounded group homomorphisms between topological rings; in this note, we replace it with $\sf Hom(X,Y)$ in compatible with \cite{KZ} for homomorphisms as well as to show their nature as a homomorphism not an operator.
\begin{definition}\rm
Let $X$ and $Y$ be two topological rings. A group homomorphism $T:X \to
Y$ is said to be
\begin{itemize}
\item[$(1)$] \emph{{\sf nr}-bounded} if there exists a
zero neighborhood $U\subseteq X$  such that $T(U)$ is bounded in $Y$;

\item[$(2)$] \emph{{\sf br}-bounded} if for every bounded set $B
\subseteq X$, $T(B)$ is bounded in $Y$.
\end{itemize}
\end{definition}

The set of all {\sf nr}-bounded ({\sf br}-bounded) homomorphisms
from a topological ring $X$ to a topological ring $Y$ is denoted
by ${\sf Hom_{nr}}(X,Y)$ (${\sf Hom_{br}}(X,Y)$). We write ${\sf
Hom}(X)$ instead of ${\sf Hom}(X,X)$.

\smallskip
Now, assume $X$ is a topological ring. The class of all ${\sf
nr}$-bounded group homomorphisms on $X$ equipped with the topology of
uniform convergence on some zero neighborhood is denoted by
${\sf Hom_{nr}}(X)$. Observe that a net $(S_{\alpha})$ of ${\sf
nr}$-bounded homomorphisms converges uniformly on a neighborhood $U$ to a homomorphism $S$  if for each neighborhood $V$  there exists an $\alpha_0$ such that for each
$\alpha\geq\alpha_0$, $(S_{\alpha}-S)(U)\subseteq V$.

\smallskip
The class of all ${\sf br}$-bounded group homomorphisms on $X$ endowed
with the topology of uniform convergence on bounded sets is denoted
by ${\sf Hom_{br}}(X)$. Note that a net $(S_{\alpha})$ of ${\sf
br}$-bounded homomorphisms uniformly converges to a homomorphism $S$
on a bounded set $B\subseteq X$ if for each zero neighborhood $V$
there is an $\alpha_0$ with $(S_{\alpha}-S)(B) \subseteq V$ for
each $\alpha\ge \alpha_0$.

\smallskip
The class of all continuous group homomorphisms on $X$ equipped with the
topology of ${\sf cr}$-convergence is denoted by ${\sf Hom_{cr}}(X)$.
A net $(S_{\alpha})$ of continuous homomorphisms ${\sf cr}$-converges
to a homomorphism $S$ if for each zero neighborhood $W$, there
is a neighborhood $U$ such that for every zero neighborhood $V$ there exists an $\alpha_0$ with
$(S_{\alpha}-S)(U)\subseteq VW$ for each $\alpha\geq\alpha_0$.

\smallskip
Note that ${\sf
Hom_{nr}}(X)$, ${\sf Hom_{br}}(X)$, and ${\sf Hom_{cr}}(X)$ form subrings
of the ring of all group homomorphisms on $X$, in which, the multiplication is given by function composition.

In contrast with the case of all bounded homomorphisms between topological groups ( considered in \cite{KZ}), there are no more relations between these classes of bounded group homomorphisms between topological rings; see \cite[Example 2.1, Example 2.2, Example 3.1]{Z1} for some examples which illustrate the situation.

\section{Main Results}
First, we prove a version of \cite[Theorem 1.10]{AB} in terms of topological $\ell$-groups.
\begin{lem}\label{1}
Suppose $G$ and $H$ are $\ell$-groups with $H$ Archimedean. Moreover, assume that $T:G_{+}\to H_{+}$ preserves the addition group operations; that is $T(x+y)=T(x)+T(y)$ holds for positive elements $x,y \in G$. Then $T$ has a unique extension to a positive group homomorphism. In addition, this extension is determined ( denoted by $T$, again) via $T(x)=T(x^{+})-T(x^{-})$.
\end{lem}
\begin{proof}
Consider the extension $S$ from $G$ into $H$ determined by  $S(x)=T(x^{+})-T(x^{-})$. Using the basic properties of $\ell$-groups (\cite[Lemma 4.1]{H}) and the proof of  \cite[Theorem 1.10]{AB}, we conclude that $S$ is additive. In order to prove that $S$ preserves the inverse operation, note that the identity $0=S(x+(-x))=S(x)+S(-x)=S(x)-S(x)$, implies that $S(-x)=-S(x)$, as we wanted.
\end{proof}
In this step, we need a type of Riesz decomposition property in $\ell$-groups; the proof relies on just addition and modulus in a Riesz space so that it can be converted without any change, using identities of \cite[Lemma 4.1]{H}. For a proof in Riesz spaces, see \cite[Theorem 1.13]{AB}.
\begin{lem}\label{2}
Suppose $|x|\leq |y_1+y_2|$ holds in an $\ell$-group $G$. Then there exist $x_1,x_2\in G$ such that $x=x_1+x_2$ and $|x_i|\leq|y_i|$. If $x$ is positive, $x_1,x_2$ can be chosen to be positive.
\end{lem}
Now, we consider a version of \cite[Theorem 1.14]{AB} assuring us under a suitable condition, the positive part of  a group homomorphism can exist.
\begin{lem}\label{3}
Let $G$ and $H$ be topological $\ell$-groups with $H$ Archimedean and  $T:G\to H$ be a homomorphism between $\ell$-groups such that $\sup\{Ty: 0\le y\leq x\}$ exists for each positive $x\in G$. Then, $T^{+}=T\vee 0$ exists and is determined via
\[T^{+}(x)=\sup\{Ty: 0\leq y\leq x\},\]
for each $x\in G_{+}$.
\end{lem}
\begin{proof}
Define $S:G_{+}\to H_{+}$ by $S(x)=\sup\{Ty: 0\leq y\leq x\}$ for each positive $x\in G$. Then, we show that $S$ is additive. Fix $u,v\in G_{+}$. For every positive $y\leq u$ and $z\leq v$, we have $T(y)+T(z)=T(y+z)\leq S(u+v)$ so that $S(u)+S(v)\leq S(u+v)$. On the other hand, if $y\leq u+v$ for a positive element $y$, by Lemma \ref{2}, there are $y_1,y_2\in G_{+}$ such that $y=y_1+y_2$, $y_1\leq u$, and $y_2\leq v$. This implies that $T(y)=T(y_1)+T(y_2)\leq S(u)+S(v)$ asserting that $S$ is additive. By Lemma \ref{1}, $S$ has an extension to a positive homomorphism (denoted by $S$) from $G$ into $H$.
Suppose for a positive homomorphism $R$, we have $T\leq R$. Fix $x\in G_{+}$. For every positive $y\leq x$, we have
$Ty\leq Ry\leq Rx$, resulting in $S\leq R$.
We see that $S=T^{+}$.
\end{proof}
Recall that a homomorphism $T:G\to H$ is said to be order bounded if it maps order bounded sets into order bounded ones. The set of all order bounded homomorphisms from $G$ into $H$ is denoted by $\sf{Hom^{b}(G,H)}$. One may justify that under group operations of homomorphisms defined in \cite{KZ} and invoking \cite[Theorem 4.9]{H}, $\sf{Hom^{b}(G,H)}$ is a group.
Now, we prove a Riesz-Kantorovich formulae for order bounded homomorphisms compatible with \cite[Theorem 1.18]{AB}. Observe that according to \cite[Remark 1]{EGZ}, not every order bounded homomorphism on a topological $\ell$-group is bounded.
\begin{thm}\label{4}
Suppose $G$ and $H$ are $\ell$-groups with $H$ Dedekind complete. Then, the group $Hom^b(G,H)$ of all order bounded homomorphisms is a Dedekind complete $\ell$-group. Moreover, $T^{+}$ is defined by
\[T^{+}(x)=\sup\{Ty:  0\leq y\leq x\},\]
for each $x\in G_{+}$.
\end{thm}
\begin{proof}
For every order bounded homomorphism $T$, note that
\[\sup\{Ty: 0\leq y\leq x\}=\sup T[0,x].\]
By Lemma \ref{3}, $T^{+}$ exists. By \cite[Lemma 4.1]{H}, $\sf {Hom^b(G,H)}$ is an $\ell$-group. To prove $\sf {Hom^b(G,H)}$ is Dedekind complete, we proceed the same line as in the proof of \cite[Theorem 1.18]{AB}. Suppose $0\leq T_{\alpha}\uparrow\leq T$ in $\sf {Hom^{b}(G,H)}$. For each $x\in G_{+}$, $S(x)=\sup\{T_{\alpha}(x)\}$ exists in $H$. The identity $T_{\alpha}(x+y)=T_{\alpha}(x)+T_{\alpha}(y)$ implies that $S$ is an additive map between positive parts. So, by Lemma \ref{1}, it has an extension to a positive homomorphism ( denoted by $S$), resulting in $T_{\alpha}\uparrow S$, as desired.
\end{proof}
\begin{rem}
Suppose $X$ is a topological ring so that a topological abelian group in its own right. Recall that a subset $B\subseteq X$ is said to be bounded ( in the sense of a topological group) if for each neighborhood $U$ of the identity, there is an $n\in \Bbb N$ with $B\subseteq nU$. Spite to the case of topological vector spaces, not every singleton in a topological group is bounded ( see \cite{Z}). Nevertheless, in many classical topological groups and also connected ones, they are bounded. Therefore, from now on, we assume that the corresponding topological groups have this mild property.
\end{rem}
\begin{prop}
Suppose $X$ is a topological ring. If $X$ is Hausdorff and every singleton in bounded in the sense of a topological group, then, so are ${\sf
Hom_{nr}}(X)$, ${\sf Hom_{br}}(X)$, and ${\sf Hom_{cr}}(X)$.
\end{prop}
\begin{proof}
First, we prove for ${\sf
Hom_{nr}}(X)$. Suppose $(T_{\alpha})$ is a net of $br$-bounded homomorphisms which converges to homomorphisms $T$ and $S$ uniformly on some zero neighborhood $U\subseteq X$. We must show that $T=S$. Assume that $W$ is an arbitrary zero neighborhood. There is a zero neighborhood $V$ with $V+V\subseteq W$. There exists an $\alpha_0$ such that $(T_{\alpha}-T)(U)\subseteq V$ and $(T_{\alpha}-S)(U)\subseteq V$ for each $\alpha\ge\alpha_0$.
Thus,
\[(T-S)(U)\subseteq (T_{\alpha}-T)(U)+(T_{\alpha}-S)(U)\subseteq V+V\subseteq W.\]
So, for each $x\in U$, we have $(T-S)(x)\in W$. Since $X$ is Hausdorff, we see that $T(x)=S(x)$. Now, for any $x\in X$, there is a positive integer $n$ with $x\in nU$. This means that there is a $y\in U$ with $x=ny$. So, by the previous procedure, we conclude that $T(x)=S(x)$, as claimed.

Now, we show that ${\sf
Hom_{br}}(X)$ is also Hausdorff. Observe that every singleton in a topological ring is bounded. Suppose $(T_{\alpha})$ is a net in ${\sf
Hom_{br}}(X)$ which is convergent to homomorphisms $T$ and $S$. Fix any $x\in X$. Assume that $W$ is an arbitrary zero neighborhood and choose zero neighborhood $V$ with $V+V\subseteq W$. There exists an $\alpha_0$ with $(T_{\alpha}-T)(x)\in V$ and $(T_{\alpha}-S)(x)\in V$ for each $\alpha\ge\alpha_0$. Thus,
\[(T-S)(x)=(T_{\alpha}-T)(x)+(T_{\alpha}-S)(x)\in V+V\subseteq W,\]
as desired.

Finally, we show that ${\sf
Hom_{cr}}(X)$ is Hausdorff. Suppose $(T_{\alpha})$ is a net in ${\sf
Hom_{cr}}(X)$ which is $cr$-convergent to homomorphisms $T$ and $S$. Choose arbitrary zero neighborhood $W$ and find zero neighborhood $V$ such that $V+V\subseteq W$. Consider zero neighborhood $V_1$ with $V_1V_1\subseteq V$. There is a zero neighborhood $U$ such that for every zero neighborhood $V_0$ we can find an index $\alpha_0$ such that $(T_{\alpha}-T)(U)\subseteq V_0V_1$ and $(T_{\alpha}-S)(U)\subseteq V_0V_1$ for any $\alpha\ge\alpha_0$.
Fix any $x\in X$. There is $n\in\Bbb N$ with $x\in nU$. Choose zero neighborhood $V_0$ with $nV_0\subseteq V_1$. There is an $\alpha_0$ such that
$(T_{\alpha}-T)(U)\subseteq V_0V_1$ and $(T_{\alpha}-S)(U)\subseteq V_0V_1$ for any $\alpha\ge\alpha_0$. Thus,
\[(T-S)(U)\subseteq (T_{\alpha}-T)(U)+(T_{\alpha}-S)(U)\subseteq V_0V_1+V_0V_1.\]
Find $y\in U$ with $x=ny$. This implies that
\[(T-S)(x)=(T-S)(ny)= (T_{\alpha}-T)(ny)+(T_{\alpha}-S)(ny)\in nV_0V_1+nV_0V_1\subseteq V_1V_1+V_1V_1\subseteq V+V\subseteq W.\]
\end{proof}
\begin{rem}
Compatible with homomorphisms on a topological $\ell$-group, not every order bounded group homomorphism between topological $\ell$-rings is bounded and vise versa.

Suppose $X={\Bbb R}^{\Bbb N}$, the ring of all sequences with product topology, coordinate-wise ordering and pointwise multiplication. Consider the identity group homomorphism $I$ on $X$. It is indeed order bounded but not $\sf nr$-bounded ( see \cite[Example 2.1]{Z1}). Moreover, if we replace pointwise multiplication in ${\Bbb R}^{\Bbb N}$ with zero one, then the identity group homomorphism is still order bounded but neither $\sf nr$- nor $ \sf br$- bounded. Suppose $X=\ell_{\infty}$ with the usual norm topology and $Y$ is $\ell_{\infty}$ with the product topology inherited from ${\Bbb R}^{\Bbb N}$; both of them, with coordinate-wise ordering and pointwise multiplication are topological $\ell$-rings. Then the identity group homomorphism from $Y$ into $X$ is order bounded but not continuous, certainly.

\end{rem}
We recall that topology $\tau$ on a topological $\ell$-ring $(X,\tau)$ is Fatou if $X$ has a base of zero neighborhoods which are order closed. Furthermore, observe that  a Birkhoff and Pierce ring ( $f$-ring) is a lattice ordered ring with this property: $a\wedge b=0$ and $c\geq 0$ imply that $ca\wedge b=ac\wedge b=0$.
\begin{lem}\label{20000}
Suppose $X$ is a Dedekind complete locally solid $f$-ring with Fatou topology and ${\sf Hom^{b}_{nr}}(X)$ is the ring of all order bounded $nr$-bounded group homomorphisms.
Then ${\sf Hom^{b}_{nr}}(X)$ is a topological $\ell$-ring.
\end{lem}
\begin{proof}
It suffices to prove that for a homomorphism $T\in {\sf Hom^{b}_{nr}}(X)$, $T^{+} \in {\sf Hom^{b}_{nr}}(X)$. By Theorem \ref{4}, for each positive $x\in X$, we have
\[T^{+}(x)=\sup \{T(u):  0\leq u\leq x\}.\]
Choose a zero neighborhood $U\subseteq X$ such that $T(U)$ is bounded. So, for arbitrary neighborhood $W$, there is a zero neighborhood $V$ with $VT(U)\subseteq W$. Therefore, for each $x\in U_{+}$ and for each $y\in V_{+}$, $yT(x)\in W$, so that using \cite[Theorem 3.15]{J}, solidness of zero neighborhoods $U,V$, and order closedness of $W$, yields that $T^{+}(U)$ is also bounded.
   Now, we show that the lattice operations are continuous. Suppose $(T_{\alpha})$  is a  net of order bounded $nr$-bounded group homomorphisms that  converges uniformly on some  zero neighborhood $U\subseteq X$ to homomorphism $T$ in ${\sf Hom^{b}_{nr}}(X)$. Choose arbitrary neighborhood $W\subseteq X$. Fix $x\in U_{+}$. Now, consider the following lattice inequality:
 \[\sup\{T_{\alpha}(u): 0\leq u\leq x\}-\sup\{T(u): 0\leq u\leq x\}\]
 \[\le\sup\{(T_{\alpha}-T)(u): 0\leq u\leq x\}.\]
 There exists an $\alpha_0$ such that $(T_{\alpha}-T)(U)\subseteq W$ for each $\alpha\geq\alpha_0$. Therefore, using the order closedness of neighborhood $W$ and solidness of neighborhood $U$, we have
\[{T_{\alpha}}^{+}(x)-{T}^{+}(x)\leq({T_{\alpha}-T})^{+}(x)\in W.\]
\end{proof}
\begin{thm}\label{601}
Suppose $X$ is a Dedekind complete locally solid $f$-ring with Fatou topology. Then ${\sf Hom^{b}_{nr}}(X)$ is a locally solid $\ell$-ring with respect to the uniform convergence topology on some zero neighborhood.
\end{thm}
\begin{proof}
In the view of Lemma \ref{20000} and \cite[Theorem 4.1]{H}, it is sufficient to show that the lattice operation $T\to T^+$ is uniformly continuous in  ${\sf Hom^{b}_{nr}}(X)$. Let $T\in {\sf Hom^{b}_{nr}}(X)$ and $x\in X_{+}$. By Theorem \ref{4}, we have
 \[T^{+}(x)=\sup\{T(u):  0\leq u\leq x\}.\]
   Now, suppose $(T_{\alpha})$ and $(S_{\alpha})$ are nets of order bounded $nr$-bounded group homomorphisms that $(T_{\alpha}-S_{\alpha})$ converges uniformly on some  zero neighborhood $U\subseteq X$ to zero. Choose arbitrary neighborhood $W\subseteq X$. Fix $x\in U_{+}$. Now, consider the following lattice inequality:
 \[\sup\{T_{\alpha}(u): 0\leq u\leq x\}-\sup\{S_{\alpha}(u): 0\leq u\leq x\}\]
 \[\le\sup\{(T_{\alpha}-S_{\alpha})(u): 0\leq u\leq x\}.\]
 There exists an $\alpha_0$ such that $(T_{\alpha}-S_{\alpha})(U)\subseteq W$ for each $\alpha\geq\alpha_0$. Therefore, using the order closedness of neighborhood $W$ and solidness of neighborhood $U$, we have
\[{T_{\alpha}}^{+}(x)-{S_{\alpha}}^{+}(x)\leq({T_{\alpha}-S_{\alpha}})^{+}(x)\in W.\]

Now, using \cite[Theorem 4.1]{H}, yields the desired result.
\end{proof}
The following lemma may be known; to our best knowledge, we could not find any proof for it; we present a proof for the sake of completeness.
\begin{lem}\label{7000}
Suppose $X$ is a locally solid $f$-ring. Then, the solid hull of a bounded set is also bounded.
\end{lem}
\begin{proof}
Suppose $B\subseteq X$ is bounded. Then, by usual definition of a solid hull, we have
\[Sol(B)=\{x\in X, \exists y\in B: |x|\leq |y|\}.\]
Let $W$ be an arbitrary zero neighborhood of $X$. There exists a zero neighborhood $V$ with $VB\subseteq W$. For each $x\in Sol(B)$, there is $y\in B$ such that $|x|\leq |y|$ so that for each $z\in V$, the inequality $|zx|= |z||x|\leq |z||y|=|zy|$ in connection with solidness of zero neighborhood $W$, imply that $V Sol(B)\subseteq W$, as we wanted.
\end{proof}
\begin{lem}\label{30000}
Suppose $X$ is a Dedekind complete locally solid $f$-ring with Fatou topology and ${\sf Hom^{b}_{br}}(X)$ is the ring of all order bounded $br$-bounded group homomorphisms.
Then ${\sf Hom^{b}_{br}}(X)$ is a topological $\ell$-ring.
\end{lem}
\begin{proof}
It suffices to prove that for a homomorphism $T\in {\sf Hom^{b}_{br}}(X)$, $T^{+} \in {\sf Hom^{b}_{br}}(X)$. By Theorem \ref{4}, we have
\[T^{+}(x)=\sup \{T(u):  0\leq u\leq x\}.\]
Fix a bounded set $B\subseteq X$. Without loss of generality, we may assume that $B$ is also solid; otherwise consider the solid hull of $B$ which is by Lemma \ref{7000}, bounded. So, for arbitrary neighborhood $W$, there is a zero neighborhood $V$ with $VT(B)\subseteq W$. Therefore, for each $x\in B_{+}$ and for each $y\in V_{+}$, $yT(x)\in W$, so that using \cite[Theorem 3.15]{J}, solidness of zero neighborhood $V$ and bounded set $B$, and order closedness of $W$, we see that $T^{+}(B)$ is also bounded.

Now, we show that the lattice operations are continuous.
Suppose $(T_{\alpha})$ is a  net of order bounded $br$-bounded group homomorphisms that converges uniformly on bounded sets to the homomorphism $T$ in ${\sf Hom^{b}_{br}}(X)$. Fix bounded set $B\subseteq X$. Choose arbitrary neighborhood $W\subseteq X$. Fix $x\in B_{+}$. By Lemma \ref{7000}, $B$ can be considered solid. Now, observe the following lattice inequality:
 \[\sup\{T_{\alpha}(u): 0\leq u\leq x\}-\sup\{T(u): 0\leq u\leq x\}\]
 \[\le\sup\{(T_{\alpha}-T)(u): 0\leq u\leq x\}.\]
 There exists an $\alpha_0$ such that $(T_{\alpha}-T)(B)\subseteq W$ for each $\alpha\geq\alpha_0$. Therefore, using the order closedness of neighborhood $W$ and solidness of bounded set $B$, we have
\[{T_{\alpha}}^{+}(x)-{T}^{+}(x)\leq({T_{\alpha}-T})^{+}(x)\in W.\]

\end{proof}
\begin{thm}\label{600}
Suppose $X$ is a Dedekind complete locally solid $f$-ring with Fatou topology. Then ${\sf Hom^{b}_{br}}(X)$ is a locally solid $\ell$-ring with respect to the uniform convergence topology on bounded sets.
\end{thm}
\begin{proof}
In the view of Lemma \ref{30000} and \cite[Theorem 4.1]{H}, it is sufficient to show that the lattice operation $T\to T^+$ is uniformly continuous in  ${\sf Hom^{b}_{br}}(X)$. Let $T\in {\sf Hom^{b}_{br}}(X)$ and $x\in X_{+}$. By Theorem \ref{4}, we have
 \[T^{+}(x)=\sup\{T(u):  0\leq u\leq x\}.\]
   Now, suppose $(T_{\alpha})$ and $(S_{\alpha})$ are  nets of order bounded $br$-bounded group homomorphisms that $(T_{\alpha}-S_{\alpha})$ converges uniformly on bounded sets to zero. Fix bounded set $B\subseteq X$. Choose arbitrary neighborhood $W\subseteq X$. Fix $x\in B_{+}$. By Lemma \ref{7000}, $B$ can be considered solid. Now, observe the following lattice inequality:
 \[\sup\{T_{\alpha}(u): 0\leq u\leq x\}-\sup\{S_{\alpha}(u): 0\leq u\leq x\}\]
 \[\le\sup\{(T_{\alpha}-S_{\alpha})(u): 0\leq u\leq x\}.\]
 There exists an $\alpha_0$ such that $(T_{\alpha}-S_{\alpha})(B)\subseteq W$ for each $\alpha\geq\alpha_0$. Therefore, using the order closedness of neighborhood $W$ and solidness of bounded set $B$, we have
\[{T_{\alpha}}^{+}(x)-{S_{\alpha}}^{+}(x)\leq({T_{\alpha}-S_{\alpha}})^{+}(x)\in W.\]

\end{proof}

\begin{lem}\label{40000}
Suppose $X$ is a Dedekind complete locally solid $f$-ring with Fatou topology and ${\sf Hom^{b}_{cr}}(X)$ is the ring of all order bounded continuous group homomorphisms.
Then ${\sf Hom^{b}_{cr}}(X)$ is a topological $\ell$-ring.
\end{lem}
\begin{proof}
It suffices to prove that for a homomorphism $T\in {\sf Hom^{b}_{cr}}(X)$, $T^{+} \in {\sf Hom^{b}_{cr}}(X)$. By Theorem \ref{4}, for any $x\in X_{+}$, we have
\[T^{+}(x)=\sup \{T(u):  0\leq u\leq x\}.\]
Choose arbitrary zero neighborhood $W$. There exists a zero neighborhood $U$ with $T(U)\subseteq W$. Therefore, for each $x\in U_{+}$, $T(x)\in  W$, so that $T^{+}(x)\in  W$ using solidness of $U$ and order closedness of $W$. Thus, we see that $T^{+}(U)\subseteq W$.

Now, we show that the lattice operations are continuous.

Suppose $(T_{\alpha})$ is a  net of order bounded continuous group homomorphisms that $cr$-converges to the homomorphism $T$ in ${\sf Hom^{b}_{cr}}(X)$. Choose arbitrary neighborhood $W\subseteq X$. There is a zero neighborhood $U\subseteq X$ such that for each zero neighborhood $V\subseteq X$ there exists an $\alpha_0$ with $(T_{\alpha}-T)(U)\subseteq VW$ for each $\alpha\geq\alpha_0$. Now, consider the following lattice inequality:
 \[\sup\{T_{\alpha}(u): 0\leq u\leq x\}-\sup\{T(u): 0\leq u\leq x\}\]
 \[\le\sup\{(T_{\alpha}-T)(u): 0\leq u\leq x\}.\]
 Therefore, using the order closedness of neighborhoods $V,W$ and solidness of zero neighborhood $U$, we have
\[{ T_{\alpha}}^{+}(x)-{T}^{+}(x)\leq ({T_{\alpha}-T})^{+}(x)\in VW.\]
This would complete the proof.
\end{proof}
\begin{thm}\label{600}
Suppose $X$ is a Dedekind complete locally solid $f$-ring with Fatou topology. Then ${\sf Hom^{b}_{cr}}(X)$ is a locally solid $\ell$-ring with respect to the $cr$-convergence topology.
\end{thm}
\begin{proof}
In the view of Lemma \ref{40000} and \cite[Theorem 4.1]{H}, it is sufficient to show that the lattice operation $T\to T^+$ is uniformly continuous in  ${\sf Hom^{b}_{cr}}(X)$. Let $T\in {\sf Hom^{b}_{cr}}(X)$ and $x\in X_{+}$. By Theorem \ref{4}, we have
 \[T^{+}(x)=\sup\{T(u):  0\leq u\leq x\}.\]
   Now, suppose $(T_{\alpha})$ and $(S_{\alpha})$ are  nets of order bounded continuous group homomorphisms that $(T_{\alpha}-S_{\alpha})$ $cr$-converges to zero. Choose arbitrary neighborhood $W\subseteq X$. There is a zero neighborhood $U\subseteq X$ such that for each zero neighborhood $V\subseteq X$ there exists an $\alpha_0$ with $(T_{\alpha}-S_{\alpha})(U)\subseteq VW$ for each $\alpha\geq\alpha_0$. Now, consider the following lattice ineqality:
 \[\sup\{T_{\alpha}(u): 0\leq u\leq x\}-\sup\{S_{\alpha}(u): 0\leq u\leq x\}\]
 \[\le\sup\{(T_{\alpha}-S_{\alpha})(u): 0\leq u\leq x\}.\]
 Therefore, using the order closedness of neighborhoods $V,W$ and solidness of zero neighborhood $U$, we have
\[{ T_{\alpha}}^{+}(x)-{S_{\alpha}}^{+}(x)\leq ({T_{\alpha}-S_{\alpha}})^{+}(x)\in VW.\]
This would complete the proof.

\end{proof}

\end{document}